\documentclass[11pt]{amsart}%
\usepackage{amscd,amsmath,latexsym,amsthm,amsfonts,amssymb,graphicx,color,geometry,hyperref}

\usepackage[utf8]{inputenc}
\usepackage{amsmath}
\usepackage{enumitem}
\usepackage{amsfonts}
\usepackage{amssymb}
\usepackage{graphicx}
\usepackage[all]{xy}
\usepackage{indentfirst}
\usepackage{tikz-cd}
\usepackage{capt-of}
\usepackage{graphicx}
\usepackage{graphics}

\usepackage{epic}
\numberwithin{equation}{section}
\usepackage{dsfont}
\providecommand{\U}[1]{\protect\rule{.1in}{.1in}}
\theoremstyle{plain}
\newtheorem{thm}{Theorem}[section]
\newtheorem{lem}[thm]{Lemma}
\newtheorem{cor}[thm]{Corollary}
\newtheorem{dfn}[thm]{Definition}
\newtheorem{prop}[thm]{Proposition}
\theoremstyle{definition}

\newtheorem{?}[thm]{Problem}

\theoremstyle{definition}
\newtheorem*{nt*}{Notation}

\newcommand{\re}{{\rm Re}\,}
\newcommand{\hb}{\mathbb{C}_+}

\newcommand{\Ber}{\mathcal{A}^2_{\alpha}(\hb)}

\newcommand{\I}{\mathrm{i}}
\newcommand{\D}{\mathrm{d}}

\newcommand {\C} {\mathbb C}

\begin{document}
\title[]{Dynamics for affine composition operators on weighted Bergman space of a half plane}
	
\author{}

\address{}  
\email{}

\author{Artur Blois}
\address{IMECC, Universidade Estadual de Campinas, Campinas, Brazil}
\email{blois@ime.unicamp.br}
\thanks{A. Blois is a PhD student at the Programa de Matemática and is supported by CAPES (Coordenação de Aperfeiçoamento Pessoal de Ensino Superior)}
    
\author{ Osmar R. Severiano}%
\address{ IMECC, Universidade Estadual de Campinas, Campinas, Brazil}
\email{osmar.rrseveriano@gmail.com}
\thanks{ O.R. Severiano is a postdoctoral fellow at the Programa de Matemática and is supported by UNICAMP (Programa de Pesquisador de Pós-Doutorado PPPD) }

\subjclass[2020]{Primary 47A16, 47B33; Secondary 37D45}

\keywords{}

\begin{abstract}
In this article, we completely characterize the positive expansive and absolutely Cesàro composition operators $C_{\phi}f=f\circ \phi$ induced by affine self-maps $\phi$ of the right half-plane $\mathbb{C}_+$ on the weighted Bergman space $\Ber$. Furthermore, we characterize which of these operators have the positive shadowing property.
\end{abstract}
	
\maketitle

\section{Introduction}
The study of the dynamics of continuous linear operators on infinite-dimensional Banach spaces has been the subject of study by several researchers in recent decades and has undergone great development (for example, see  \cite{Frederic, Grivaux, Grosse}). The notions of \textit{expansivity} and \textit{positive shadowing property} play important roles in dynamical systems, including topological dynamics, differentiable dynamics and ergodic theory, see \cite{Pilyugin, Pilyugin1}, for instance.

Let $\mathcal{S}$ be a space of functions on a subset $\Omega$ of the complex plane $\mathbb{C}.$ A \emph{composition operator} $C_\phi$ on $\mathcal{S}$ is an operator acting by composition to the right with a chosen self-map $\phi$ of $\Omega$, that is
\[C_{\phi}f = f \circ \phi, \quad f \in \mathcal{S}.\]
Linear dynamics of composition operators has been studied on several holomorphic function spaces \cite{Bayart, Darji, Bourdon,Chalendar, Gallardo}. In \cite{Carlos}, Álvarez and Henríquez-Amador studied the affine composition operators on the Hardy space of the open right half-plane $H^2(\hb)$, that is, the bounded composition operators induced by symbols of the form 
\begin{align}\label{symbol}
\phi(w)=aw+b, \text{ where } a > 0 \text{ and } \re(b) \geq 0,
\end{align}
and completely characterized those operators that are uniformly expansive, positive expansive, uniformly positive expansive, Li-Yorke chaotic and have the positive shadowing property. In the proof of  \cite[Theorem 7]{Carlos} the authors use the argument that the reproducing kernels are dense in $H^2(\hb)$, however this claim is false, and as a consequence of this fact theys conclude that if $\phi$ is as in \eqref{symbol} then $C_{\phi}$ is uniformly positive expansive on $H^2(\hb)$ if and only if $a\in (0,1).$ As we will see in Proposition \ref{uniformly positive expansive}, although the proof \cite[Theorem 7]{Carlos} is wrong, its result holds.

Our goal in this work is  to characterize completely the affine composition operators  on the weighted Bergman space of the open right half-plane $\Ber,$ for $\alpha>-1,$ that are expansive, uniformly expansive, positive expansive, uniformly positive expansive, Li-Yorke chaotic and has positive shadowing property. Since $H^2(\hb)$ can often be formally interpreted as the ``limit case'' of the weighted Bergman spaces as $\alpha\rightarrow -1,$ that is, $H^2(\hb)=\mathcal{A}_{-1}^2(\hb)$ (see \cite[page 819]{Riikka}) our results can be seen as extensions of those proved in \cite{Carlos}.

\section{Preliminaries}\label{s1}
Throughout the work we will use the following notations: $\mathbb{N}$ denotes the set of all positive integers, $\mathbb{R}_+$ the non-negative real numbers and $\hb=\{z\in \mathbb{C}:\re(b)>0\}$ the open right half-plane. Moreover, $\mathcal{B}(X)$ denotes the space of all bounded linear operators on a complex Banach space $X.$

\subsection{The weighted Bergman space \texorpdfstring{$\Ber$}{}}
For $\alpha> -1,$ the \textit{weighted Bergman space} $\mathcal{A}^2_{\alpha}(\hb)$ consists of all analytic functions $f:\hb\rightarrow \mathbb{C}$ for which 
\begin{align}\label{norm}
\|f\|= \left(\frac{1}{\pi}\int_{-\infty}^{\infty}\int_{0}^{\infty} |f(x+\I y)|^2x^{\alpha}\D x \D y\right)^{1/2}<\infty .
\end{align}
The space $\mathcal{A}^2_{\alpha}(\hb)$ is a reproducing kernel Hilbert space when endowed with the norm defined in \eqref{norm}. The functions $\{k_w^{\alpha}:w\in \mathbb{C}_+\}$ defined by
\begin{align*}
k_{w}^{\alpha}(z)=\frac{2^{\alpha}(\alpha+1)}{(z+\overline{w})^{\alpha+2}}, \quad z\in \hb,
 \end{align*}
are the reproducing kernels for $\mathcal{A}^2_{\alpha}(\hb).$ These kernel functions have the fundamental property $\langle f, k_w\rangle=f(w)$ for all $f\in \Ber$ and $w\in \hb.$ Such spaces are studied by Ellitot and Wynn in \cite{Elliot}.

Let $\mathrm{d}\mu_{\alpha}=\frac{\Gamma(1+\alpha)}{2^{\alpha}t^{\alpha+1}}\mathrm{d}t,$ where $\Gamma$ is the Gamma function and $\mathrm{d}t$ the Lebesgue measure. According to \cite[Lem. 3.3]{Elliot}, the space $\mathcal{A}^2_{\alpha}(\hb)$ is isometrically isomorphic, via the Laplace
transform $\mathcal{L},$ to the space $L^2(\mathbb{R}_+, \mu_{\alpha}).$  That is, $f\in \Ber$ if and only if there exists a function $F\in L^2(\mathbb{R}_+, \mu_{\alpha})$ such that
\begin{align}\label{paley-wiener}
f(w)=(\mathcal{L}F)(w)=\int^{\infty}_0F(t)e^{-wt}\D t, \quad w\in \hb.
\end{align}

\subsection{Affine composition operators}
In \cite[Theorem 3.4]{Elliot}, Elliot and Wynn showed that if $\phi$ is an analytic self-map of $\hb,$ then $C_{\phi}$ is bounded on $\Ber$ if and only if $\phi$ has a finite angular derivative at the fixed point $\infty, $ that is, if $\phi(\infty)=\infty$ and if the non-tangential limit 
\begin{align}
\phi'(\infty)=\lim_{n\rightarrow \infty}\frac{w}{\phi(w)}
\end{align}
exists and is finite. In this case, the norm of $C_{\phi}$ is given by $\|C_{\phi}\|=\phi'(\infty)^{\frac{\alpha+2}{2}}.$ Moreover, if $C_{\phi}$ is a bounded composition operator on $\Ber,$ then $C_{\phi}^*$ acts on reproducing kernel as follows  $C_{\phi}^*k_w=k_{\phi(w)},$ for each $w\in \hb.$

Since the only linear fractional self-maps of $\hb$ that fix $\infty$ are precisely of affine maps of the form
\begin{align}\label{affine symbol}
\phi(w)=a w+b,  
\end{align}
where $a>0$ and $\re(b)\geq 0,$ it follows from boundedness criterion of composition operators on $\Ber$ that the only bounded composition operator induced by linear fractional self-maps of $\hb$ are those induced by symbols as in \eqref{affine symbol}. For this reason, in this work, we will call such operators of affine composition operators. A simple computation reveals that if $\phi$ is as in \eqref{affine symbol} then $n$-iterate of $\phi$ is given by
\begin{align}\label{iterate}
\phi^{[n]}(w)= \begin{cases}w+nb, & \text{if} \ a=1, \\ 
a^{n} w+\frac{\left(1-a^{n}\right)}{1-a}b,  & \text{if} \ a \neq 1 
\end{cases}
\end{align}	
and $C_{\phi}^n=\C_{\phi^{[n]}}.$  The only invertible bounded composition operators on $\Ber$ are those induced by symbols as in \eqref{affine symbol} where $a>0$ and $\re(b)=0.$ According to \cite[Theorem 3.1]{Chalendar}, the spectrum of affine composition operators on $\Ber$ are given as follows.

\begin{thm}\label{spectrum parabolic} Let $\phi(w)=w+b$ with $\re(b)\geq 0$ and $b\neq 0.$ Then the spectrum of $C_{\phi}$ acting on $\Ber$ is
\begin{enumerate}[label=\textnormal{(\roman*)}]
\item $\sigma(C_{\phi})=\mathbb{T},$ when $b\in \mathrm{i}\mathbb{R};$
\item $\sigma(C_{\phi})=\left\{e^{\I bt}:t\in [0, +\infty)\}\cup\{0\right\}.$
\end{enumerate}
\end{thm}

\begin{thm}\label{spectrum hyperbolic} Let $\phi(w)=a w+b$ where $a \in (0,1)\cup(1,+\infty)$ and $\re(b)\geq 0.$ Then the spectrum of $C_{\phi}$ acting on $\Ber$ is
\begin{enumerate}[label=\textnormal{(\roman*)}]
\item $\sigma(C_{\phi})=\left\{\lambda \in \mathbb{C}:|\lambda|=a ^{-\left(\frac{\alpha+2}{2}\right)}\right\},$ when $b\in \mathrm{i}\mathbb{R};$
\item $\sigma(C_{\phi})=\left\{\lambda \in \mathbb{C}:|\lambda|\leq a ^{-\left(\frac{\alpha+2}{2}\right)}\right\},$ when $b\in \hb.$
\end{enumerate}
\end{thm}

\subsection{Linear dynamics} Let us introduce the notions relevant to our work. Such notions will be studied in subsequent section.

\begin{dfn}
An operator $T\in \mathcal{B}(X)$ is \textit{Li-Yorke chaotic} if  there exists $x\in X$ such that
\begin{align*}
\displaystyle\liminf_{n\to \infty}\|T^{n}x\|=0 \quad \text{and} 
\quad \displaystyle\limsup_{n\to \infty}\|T^{n}x\|=\infty.
\end{align*}
In this case, such vector $x$ is called an \textit{irregular vector} for $T.$ 
\end{dfn}
It is immediate that isometric operators does not admit irregular vectors, and hence they are not Li-Yorke chaotic.

\begin{dfn} An invertible operator $T\in \mathcal{B}(X)$ is:
\begin{enumerate}[label=\textnormal{(\roman*)}]
\item \textit{Expansive} if for each $x\in X$ with $\|x\|=1,$ there exists $n\in \mathbb{Z}$ such that $\|T^nx\|\geq 2.$
\item \textit{Uniformly expansive}, if there exists $n\in \mathbb{N}$ such that for each $x\in X$ with $\|x\|=1,$ we have $\|T^nx\|\geq 2$ or  $\|T^{-n}x\|\geq 2.$
\end{enumerate}
\end{dfn}

If $T\in \mathcal{B}(X)$ is not necessarily invertible, it is defined the notions of \textit{positively expansive} and \textit{uniformly positively expansive}, in this case we use the appropriate modifications, putting $\mathbb{N}$ instead of $\mathbb{Z}.$

\begin{dfn} Let $T\in \mathcal{B}(X)$ and $\delta>0.$ A sequence $(x_n)_{n\in \mathbb{N}}$ on $X$ is said to be a positive $\delta-$pseudo orbit of $T$ if $\|T x_{n} - x_{n+1}\| \leq \delta$ for all $n\in \mathbb{N}.$
\end{dfn}
If $x\in X$ is such that $Tx\neq 0,$ then the sequence $(x_{n})_{n\in \mathbb{N}}$ defined by
\begin{align}\label{pseudo}
x_n=\frac{\delta}{\|Tx\|}\sum_{j=1}^{n-1}T^{n-j}x
\end{align}
is a positive $\delta-$pseudo orbit of $T.$ Indeed
\begin{align*}
\|Tx_n-x_{n+1}\|=\frac{\delta}{\|Tx\|}\|\sum_{j=1}^{n-1}T^{n-j+1}x-\sum_{j=1}^{n-1}T^{n-j}x\|=\frac{\delta}{\|Tx\|}\|Tx\|=\delta.
\end{align*}

\begin{dfn} An operator $T\in \mathcal{B}(X)$ is said to have the positive shadowing property if for every $\epsilon>0$ there exists $\delta > 0$ such that every positive $\delta-$pseudo orbit  $(x_{n})_{n\in \mathbb{N}}$ of $T$ is $\epsilon-$shadowed, i.e., there exists $x\in X$ such that  $\|T^{n}x - x_{n}\| \leq \epsilon$ for all $n\in \mathbb{N}$
\end{dfn}

In \cite{Bonilla}, Bernardes, Bonilla and Peris studied mean Li-Yorke chaos in Banach space, in particular there is presented a notion of \textit{absolutely Cesàro bounded} which we introduce now.

\begin{dfn} An operator $T\in \mathcal{B}(X)$ is said to be absolutely Cesàro bounded if there exists a constant $M>0$ such that 
\begin{align*}
\sup_{n\in \mathbb{N}}\left\lbrace\frac{1}{n}\sum_{j=1}^n\|T^jx\|\right\rbrace\leq M\|x\|, 
\end{align*}
for all $x\in X.$
\end{dfn}
\section{Main results}
We begin by showing that affine composition operators on $\Ber$ do not admit irregular vectors, and hence that such operators are not Li-Yorke chaotic. Before proceeding, we need the following auxiliary results, which establish some estimates that will be used frequently.

\begin{prop}\label{estimates} Let $\phi(w)=a w+b$ with $a>0$ and $\re(b)\geq0.$ For $f\in \Ber$ we have 
\begin{enumerate}[label=\textnormal{(\roman*)}]
\item $\|C_{\phi}f\|=a^{-\left(\frac{\alpha+2}{2}\right)}\|f\|,$ when $\re(b)=0;$
\item $\|C_{\phi}f\|\leq a ^{-\left(\frac{\alpha+2}{2}\right)}\|f\|.$
\end{enumerate}
\end{prop}

The proof of the proposition above follows from straightforward computations and it is
left to the reader.
\begin{lem}\label{lower estimates}
Let $\phi(w)=aw+b$ with $a\in (0,1)$ and $\re(b)\geq 0.$ For each $f\in \Ber$ there exists $\delta>0$ such that 
\begin{align*}
\|C_{\phi}^nf\|\geq \delta a ^{-\left(\frac{\alpha+2}{2}\right)n}\|f\|
\end{align*}
for sufficiently large $n.$  
\end{lem}
\begin{proof}
For each $n,$ we consider the auxiliary symbols $\varphi_n$ and $\psi_n$ defined as follow
\begin{align*}
\varphi_n(w)=a ^nw  \quad \text{and} \quad \psi_n(w)=w+\frac{1-a ^n}{1-a }b
\end{align*}
then $\phi^{[n]}=\psi_n\circ \varphi_n.$ Let $f\in \Ber$ be a non-zero vector and let $g=\|f\|^{-1}f.$ By Proposition \ref{estimates}, we have
\begin{align*}
\|C_{\phi}^ng\|=\|C_{\varphi_n}C_{\psi_n}g\|= a^{-\left(\frac{\alpha+2}{2}\right)n}\|C_{\psi_n}g\|.
\end{align*}
Now we obtain a lower estimate for $\|C_{\psi_n}g\|.$ For each $w\in \hb,$ the Cauchy Schwarz inequality gives
\begin{align*}
\|C_{\psi_n}g\|\|k_w\|\geq |\langle C_{\psi_n}g, k_w\rangle|=|\langle g, C_{\psi_n}^*k_w\rangle|=|\langle g, k_{\psi_n(w)}\rangle|=|g(\psi_n(w))|
\end{align*}
for all $n.$ Since $\psi_n(w)\rightarrow w+\frac{b}{1-a}$ as $n\rightarrow \infty,$ it follows that $|g(\psi_n(w))|\rightarrow |g(w+\frac{b}{1-a})|$ as $n\rightarrow \infty.$ Let $z_0\in \hb$ such that $g(z_1)\neq 0,$ where $z_1=z_0+\frac{b}{1-a}.$ Then for sufficiently large $n,$ we get
$|g(\psi_n(z_0))|>\frac{|g(z_1)|}{2}.$ If $\delta=\frac{|g(z_1)|}{2\|k_{z_0}\|},$ we obtain
\begin{align*}
\|C_{\phi}^nf\|> \delta a ^{-\left(\frac{\alpha+2}{2}\right)n}\|f\|
\end{align*}
for sufficiently large $n.$
\end{proof}

\begin{prop}
Let $\phi(w)=a w+b$ with $a >0$ and $\re(b)\geq0$, then $C_{\phi}$ does not admit an irregular vector on $\Ber.$
\end{prop}
\begin{proof}  Let $f\in \Ber$ be a non-zero vector. By Proposition \ref{estimates},  $\|C_{\phi}^nf\|\leq a ^{-\left(\frac{\alpha+2}{2}\right)n}\|f\|$ for all $n\in \mathbb{N}.$ If $a\geq 1,$ then $(\|C_{\phi}^nf\|)_{n\in \mathbb{N}}$ is bounded and hence $f$ is not an irregular vector for $C_{\phi}$. If $a\in (0,1)$ then Lemma \ref{lower estimates} give $\|C_{\phi}^nf\|\geq \delta a ^{-\left(\frac{\alpha+2}{2}\right)n}\|f\|$ for sufficiently large $n.$  Since $a^{-\left(\frac{\alpha+2}{2}\right)n}\rightarrow \infty$ as $n\rightarrow \infty,$ it follows that $\|C_{\phi}^nf\|\geq 2$ for sufficiently large $n.$ In this case, $\displaystyle\liminf_{n\to \infty}\|C_{\phi}^{n}f\|\geq 2$ and hence $f$ is not an irregular vector for $C_{\phi}$.
\end{proof}

\begin{cor}
Let $\phi(w)=a w+b$ with $a >0$ and $\re(b)\geq0$, then $C_{\phi}$ is not Li-Yorke chaotic on $\Ber.$
\end{cor}
Now we deal with the notions of expansivity.

\begin{prop}\label{uniformly expansive}
Let $\phi(w)=a w+b$ with $a >0$ and $\re(b)=0.$ Then the following statements are equivalent
\begin{enumerate}[label=\textnormal{(\roman*)}]
\item $C_{\phi}$ is uniformly expansive on $\Ber;$
\item  $C_{\phi}$ is expansive on $\Ber;$
\item $a\neq 1.$
\end{enumerate}
\end{prop}

\begin{proof}
(i) $\implies$ (ii) Follows from the definitions. (ii) $\implies$ (iii) Now assume $a=1.$ Since $a=1$ and $\re(b)=0,$ it follows from Proposition \ref{normal} that $C_{\phi}$ is a unitary operator and hence $C_{\phi}$ is not expansive. (iii) $\implies$ (i) By Proposition \ref{estimates}, we  obtain
\begin{align}\label{inequality}
\|C_{\phi}^nf\|=a ^{-\left(\frac{\alpha+2}{2}\right)n}\|f\| \ \quad \text{and} \quad \|(C_{\phi}^{-1})^nf\|=a ^{\left(\frac{\alpha+2}{2}\right)n} \|f\|,
\end{align}
for all $f\in \Ber$ and $n\in \mathbb{N}.$ If $a\in (0,1),$ then there exists $j\in \mathbb{N}$ such that $a ^{-\left(\frac{\alpha+2}{2}\right)j} \geq 2,$ while if $a \in (1, +\infty),$ then there exists $\ell\in \mathbb{N}$ such that $a ^{\left(\frac{\alpha+2}{2}\right)\ell} \geq 2.$  By \eqref{inequality}, for all $f\in \Ber$ with $\|f\|=1$ we obtain $\|C_{\phi}^jf\|\geq 2$ or $\|(C_{\phi}^{-1})^{\ell}f\|\geq 2.$ Hence $C_{\phi}$ is uniformly expansive. 
\end{proof}

\begin{prop}\label{uniformly positive expansive}
Let $\phi(w) = a w+b$ with $a >0$ and $\re(b) \geq 0$. Then the following statements are equivalent
\begin{enumerate}[label=\textnormal{(\roman*)}]
\item $C_{\phi}$ is uniformly positive expansive on $\Ber;$
\item  $C_{\phi}$ is positive expansive on $\Ber;$
\item $a\in (0,1).$
\end{enumerate}
\end{prop}

\begin{proof} (i) $\implies$ (ii) Follows from the definitions. (ii) $\implies$ (iii) Assume $a\geq1$ then Proposition \ref{estimates} gives $\|C^{n}_{\phi}f\| \leq a ^{-\left(\frac{\alpha+2}{2}\right)n}\leq 1$ for all $f\in \Ber$ with $\|f\|=1$ and $n\in\mathbb{N}.$ Hence $C_{\phi}$ is not positive expansive. (iii) $\implies$ (i) Let $a\in (0,1).$ Suppose that $C_{\phi}$ is not uniformly positive expansive to reach a contradiction. By the assumption, we can find $f\in \Ber$ with $\|f\|=1$ such that $\|C_{\phi}^nf\|<2$ for all $n\in \mathbb{N}.$ On the other hand, since $a\in (0,1),$ it follows from Lemma \ref{lower estimates} that we can find $n$ such that $\|C_{\phi}^nf\|>2.$ This contradiction implies that $C_{\phi}$ must be uniformly positive expansive.
\end{proof}

Recall that an operator $T\in \mathcal{B}(X)$ is said be \textit{hyperbolic} if $\sigma(T)\cap \mathbb{T}.$ Bernardes \textit{et al.} studied expansivity and shadowing in linear dynamics \cite{Bernardes}. They proved the following results for $T\in \mathcal{B}(X):$
\begin{itemize}
\item  If $T$ is hyperbolic then $T$ has the positive shadowing property \cite[Theorem 13]{Bernardes}.
\item If $T$ is normal (here we assume $X$ a Hilbert space), then $T$ has the positive shadowing property if and only if $T$ is hyperbolic \cite[Theorem 30]{Bernardes}.
\end{itemize}  

In order to use \cite[Theorem 30]{Bernardes}, we first classify the affine composition operators on $\Ber$ that are normal. The version of Proposition \ref{normal} for $\Ber$ when $\alpha$ is an positive integers appears in \cite[Theorem 8.7]{HS}.

\begin{prop}\label{normal} Let $\phi(w)=a w+b$ with $a>0$ and $\re(b)\geq 0.$ Then $C_{\phi}$ is normal on $\Ber$ if and only if $a=1$ or $\re(b)=0.$
\end{prop}
\begin{proof}We first will prove that $C_{\phi}$ is unitarily equivalent via the Laplace transform $\mathcal{L}$ to the operator $\widehat{C}_{\phi}:L^2(\mathbb{R}_+, \mu_{\alpha})\rightarrow L^2(\mathbb{R}_+, \mu_{\alpha})$ defined by $(\widehat{C}_{\phi}F)(t)=\frac{1}{a}e^{-bt/a}F(t/a).$ For $F\in L^2(\mathbb{R}_+, \mu_{\alpha})$ and $w\in \hb,$ the Laplace transform gives 
\begin{align*}
(C_{\phi}\mathcal{L}F)(w)&=(\mathcal{L}F)(\phi(w))=\int^{\infty}_0F(t)e^{-(a w+b)t}\D t=\int^{\infty}_0F(t)e^{-a wt }e^{-bt}\D t\\
&=\int_0^{\infty}\frac{1}{a}e^{-bs/a}F\left(s/a\right)e^{-ws}\D s=\int^{\infty}_0(\widehat{C}_{\phi}F)(s)e^{-ws}\D s=(\mathcal{L}\widehat{C}_{\phi}F)(w),
\end{align*}
hence $C_{\phi}\mathcal{L}=\mathcal{L}\widehat{C}_{\phi}.$  Moreover, for $F,G\in L^2(\mathbb{R}_+, \mu_{\alpha}),$ we obtain
\begin{align*}
\frac{2^{\alpha+1}}{2\pi\Gamma(\alpha+1)}\langle \widehat{C}_{\phi}F,G\rangle=
\int^{\infty}_0\frac{1}{a}e^{-bt/a}F(t/a)\overline{G(t)}t^{-(\alpha+1)}\D t
=\int^{\infty}_0e^{-bs}F(s)\overline{G(a s)}(a s)^{-(\alpha+1)}\D s
\end{align*}
which gives  $(\widehat{C}_{\phi}^*F)(t)=a^{-(\alpha+1)}e^{-\overline{b}t}F(a t).$
Since $C_{\phi}$ and $\widehat{C}_{\phi}$ are unitarily equivalent it is enough to check when $\widehat{C}_{\phi}$ is normal. From what we have done so far, we get
\begin{align}\label{normal equation}
(\widehat{C}_{\phi}\widehat{C}_{\phi}^*F)(t)=a^{-(\alpha+2)}e^{-2\re(b)t/\mu}F(t) \quad \text{and}  \quad (\widehat{C}_{\phi}^*\widehat{C}_{\phi}F)=a^{-(\alpha+2)}e^{-\re(b)t}F(t)
\end{align}
for each $F\in L^2(\mathbb{R}_+, \mu_{\alpha}).$ Hence, $\widehat{C}_{\phi}$ is normal if and only if $e^{-2\re(b)t/a}=e^{-2\re(b)t}.$ This last equality holds if and only if $a=1$ or $\re(b)=0.$
\end{proof}

The proof of the next result is well known (see for example \cite[Prop. 1]{Carlos}), and so is left to the reader

\begin{lem}\label{lemma shadowing} Let $C_{\phi}$ be a bounded composition operator on $\Ber.$ If $\phi$ has a fixed point on $\hb$ then $C_{\phi}$ does not have the positive shadowing property.
\end{lem}

\begin{thm}\label{shadowing} Let $\phi(w)=a w+b$ with $a>0$ and $\re(b)\geq 0.$ Then $C_{\phi}$ has the positive shadowing property on $\Ber$ if and only if one of the following cases occurs: 
\begin{enumerate}[label=\textnormal{(\roman*)}]
\item $a\in (0,1)$ and $\re(b)=0;$
\item $a>1.$ 
\end{enumerate}
\end{thm}

\begin{proof} We first consider $a=1.$ In this case, $C_{\phi}$ is normal on $\Ber$ (see Proposition \ref{normal}) and is not hyperbolic (see Theorem \ref{spectrum parabolic}). By \cite[Theorem 30.]{Bernardes}, $C_{\phi}$ does not have the positive shadowing property. Now let $a\in (0,1).$ In addition, if $\re(b)=0$ then $C_{\phi}$ is normal on $\Ber$ and is hyperbolic (see Theorem \ref{spectrum hyperbolic}). By \cite[Theorem 30.]{Bernardes}, $C_{\phi}$ has the positive shadowing property. If $\re(b)>0$ then $\phi$ has a fixed point in $\hb.$ By Lemma \ref{lemma shadowing} follows that $C_{\phi}$ does not have the positive shadowing property. Finally, assume $a>1.$ Then $C_{\phi}$ is hyperbolic (see Theorem \ref{spectrum hyperbolic}). By \cite[Theorem 13.]{Bernardes} follows that $C_{\phi}$ has the positive shadowing property.
\end{proof}

We finished this work by classifying the affine composition operators on $\Ber$ that are absolutely Cesàro bounded.

\begin{prop}
Let $\phi(w)=a w+b$ with $a > 0$ and $\re(b) \geq 0$. Then $C_{\phi}$ is an absolutely Cesàro bounded operator on $\Ber$ if and only if $a \geq 1.$
\end{prop}

\begin{proof} For simplicity, we write $c= a^{-\left(\frac{\alpha+2}{2}\right)}$. By Proposition \ref{estimates}, $\|C^{j}_{\phi}f\|\le c^{j}\|f\|$ for all $f\in \Ber$ and $j\in \mathbb{N}.$ If $a=1$ we have $\displaystyle \sum_{j=1}^n\|C_{\phi}^jf\|\leq n\|f\|$  and if $a>1$ then $(c^j)_{j\in \mathbb{N}}$ forms a convergent geometric series. Hence, taking $M\geq \max\{1, c/(1-c)\},$ we obtain
\begin{equation*}
\sup_{n \in \mathbb N}\left\lbrace  \dfrac{1}{n}\sum_{j=1}^{n}\|C^j_{\phi}f\|\right\rbrace\ \leq \sup_{n \in \mathbb{N}}\left\lbrace \dfrac{1}{n}\sum_{j=1}^{n}c^{j}\right\rbrace\|f\| \leq M\|f\|,
\end{equation*}
for each $f\in \Ber.$ If $a \in (0,1),$ then the sequence $(c^{n}/n)_{n\in \mathbb{N}}$ is unbounded. By Lemma \ref{lower estimates} there exists $\delta > 0$ such that $\|C^{n}_{\phi}k_1\| \geq c^{n}\delta\|k_1\|$ for sufficiently large $n \in \mathbb{N}.$ Taking sufficiently large $n,$ we get
\begin{align*}
\frac{1}{n}\sum_{j=1}^{n}\|C^{j}_{\phi}k_1\|\geq \frac{1}{n}\|C_{\phi}^nk_1\|\geq \delta \frac{c^n}{n}\|k_1\|
\end{align*}
which implies 
 \begin{align*}
\sup_{n\in \mathbb{N}}\left\lbrace \frac{1}{n}\sum_{j=1}^{n}\|C^{j}_{\phi}k_1\|\right\rbrace=\sup_{n\in\mathbb{N}}\left\lbrace\frac{c^{n}}{n}\right\rbrace\|k_1\|=\infty.
\end{align*}
Therefore, $C_{\phi}$ is not an absolutely Cesàro bounded operator on $\Ber.$
 \end{proof}
 



\end{document}